\documentclass[12pt]{amsart}
\usepackage{mathpazo}
\usepackage{hyperref}
\usepackage{bm}
\usepackage{comment}
\usepackage[margin=1in]{geometry}

\usepackage{graphicx}
\usepackage{tikz-cd}

\usepackage[cmtip, all]{xy}
\usepackage{array}

\usepackage{amssymb,amsmath,latexsym,amsthm}
\newtheorem{theorem}{Theorem}[section]
\newtheorem{lemma}[theorem]{Lemma}
\newtheorem{proposition}[theorem]{Proposition}
\newtheorem*{theorem*}{Theorem}

\newtheorem{predefinition}[theorem]{Definition}
\newenvironment{definition}{\begin{predefinition}\rm}{\end{predefinition}}
\newtheorem{preremark}[theorem]{Remark}
\newenvironment{remark}{\begin{preremark}\rm}{\end{preremark}}
\newtheorem{prenotation}[theorem]{Notation}

\newtheorem{preexample}[theorem]{Example}

\newtheorem{preclaim}[theorem]{Claim}

\newtheorem{prequestion}[theorem]{Question}

\newtheorem{preapplication}[theorem]{Application}

\numberwithin{equation}{section}

\newcommand \ZZ {{\mathbb Z}}

\newcommand  \FF {{\mathbb F}}

\newcommand \PP {{\mathbb P}^1}

\newcommand \CC {{\mathbb C}}

\title{Vanishing of Dirichlet L-functions at the central point over function fields}
\date{}

\author{Ravi Donepudi}
\address{Department of Mathematics, University of Illinois at Urbana-Champaign, Urbana, IL 61801, USA}
\email{donepud2@illinois.edu}

\author{Wanlin Li}
\address{Department of Mathematics, Massachusetts Institute of Technology,
Cambridge, MA 02139, USA}
\email{wanlinli@mit.edu}

\begin{document}

\maketitle

\begin{abstract}
We give a geometric criterion for Dirichlet $L$-functions associated to cyclic characters over the rational function field $\mathbb{F}_q(t)$ to vanish at the central point $s=1/2$. The idea is based on the observation that vanishing at the central point can be interpreted as the existence of a map from the projective curve associated to the character to some abelian variety over $\mathbb{F}_q$. Using this geometric criterion, we obtain a lower bound on the number of cubic characters over $\mathbb{F}_q(t)$ whose $L$-functions vanish at the central point where $q=p^{4n}$ for any rational prime $p \equiv 2 \bmod 3$. We also use recent results about the existence of supersingular superelliptic curves to deduce consequences for the $L$-functions of Dirichlet characters of other orders.
\end{abstract}

\section{Introduction}

Let $L(s,\chi)$ be the $L$-function attached to a Dirichlet character $\chi$. The behavior of these $L$-functions in the critical strip $\{s\in \CC|0<\Re(s)<1 \}$, specifically the locations of the zeroes of $L(s,\chi)$, has been a subject of intense study in number theory. In particular, under the Generalized Riemann Hypothesis (GRH), it is expected that the only points in the critical strip where $L(s,\chi)$ vanishes are on the vertical line $\Re (s) = \frac{1}{2}$. 

In a different direction, a conjecture of Chowla \cite{Chowla} predicts that the $L$ function $L(s,\chi)$ associated to a quadratic character $\chi$ does \emph{not} vanish for any real $s\in (0,1)$, and in particular $L(\frac{1}{2},\chi)\neq 0$ for any $\chi$. More generally, it is expected that no Dirichlet $L$-function of any order vanishes at $s=\frac{1}{2}$. In this classical setting, no counterexamples to Chowla's conjecture have been found. Results of Soundararajan \cite{Sound00} and Baluyot-Pratt \cite{BP18} establish lower bounds on the number of Dirichlet characters (ordered by conductor) that do not vanish at $s=\frac{1}{2}$. This supports the above expectation. In this article, we study the analog of this conjecture over function fields in positive characteristic. 

The arithmetic of rational functions in one variable over a finite field bears many similarities to the theory of number fields as they are both instances of the general theory of global fields. So analogs of virtually every phenomenon of the latter setting to the former have been investigated. Specifically, the above questions regarding vanishing of $L$-functions can be asked with replacing $\mathbb Q$ by $\mathbb F_q(t)$, the field of rational functions over the finite field $\mathbb F_q$. While GRH is still open in the classical case of Dirichlet $L$-functions, it (and much more) is proven in the case of $L$-functions attached to characters over function fields as part of the Weil conjectures, which are now theorems.

Over function fields in positive characteristic, $L$-functions attached to quadratic Dirichlet characters are exactly the $L$-functions attached to hyperelliptic curves. In \cite{BF} by Bui--Florea, the authors investigate the analog of Chowla's conjecture over function fields and prove that over a finite field of odd cardinality, the proportion of hyperelliptic $L$-functions that vanish at $s=\frac{1}{2}$ is at most $0.057$. In \cite{Wanlin}, the second author showed that the analog of Chowla's conjecture in function fields does not hold and gave an explicit lower bound on the number of counter examples with bounded genus. Even though these results show that there are infinitely many hyperelliptic $L$-functions that vanish at $s=\frac{1}{2}$, it is believed that $100\%$ of hyperelliptic $L$-functions do not vanish at $s=\frac{1}{2}$ in a sense that will be made precise in Section $2$. The work of Ellenberg--Li--Shusterman \cite{ELS} provides evidence supporting this belief. Extending this work beyond hyperelliptic $L$-functions is the recent work of David-- Florea--Lal\'in \cite{DaFloLa2}, where the authors show a positive proportion of L-functions associated to cubic characters do not vanish at the critical point $s=1/2$ based on their previous work \cite{DaFloLa} in which they study moments of the central value of cubic L-functions. 

The strategy used in \cite{Wanlin} stems from the observation that the existence of a quadratic Dirichlet character over $\mathbb{F}_q(t)$ whose $L$-function vanishes at the central point $s=\frac{1}{2}$ is equivalent to the existence of a hyperelliptic curve over the finite field $\mathbb{F}_q$ which admits $\sqrt{q}$ as an eigenvalue for the Frobenius action on the $\ell$-adic Tate module of its Jacobian. Moreover, all hyperelliptic curves which admit a dominant map to this curve induce a quadratic character whose $L$-function also vanishes at $s=\frac{1}{2}$. Thus, a large part of the work is to prove a lower bound on the number of curves in certain families which admit dominant maps to a fixed curve.

While the above article only deals with hyperelliptic $L$-functions, the machinery developed and strategies used are applicable to more general $L$-functions. In this article, we extend this work by considering $L$-functions attached to characters associated to geometric cyclic covers of the projective line, also commonly referred to as superelliptic curves due to the similarity of their defining equations with elliptic curves. See Definition \ref{superelliptic} for the precise definition of superelliptic curves. We first prove a lower bound on the number of $\ell$-th order superelliptic curves admitting a dominant map to a fixed curve when $\ell$ satisfies some congruence restrictions necessitated by our liberal use of Kummer theory:

\begin{theorem*}[Theorem \ref{GeneralThm}]

	Let $\ell$  be an odd prime coprime to $q$. Let $C_0$ be an $\ell$-th order superelliptic curve of genus $g$ defined over $\mathbb{F}_q$ with affine equation $y^\ell=\prod_{i=1}^{\ell-1}f_i^i$ where the $f_i$ are pairwise coprime, squarefree polynomials of degree $d_i$ each and $d=\sum_{i=1}^{\ell-1} d_i$. Take a model which satisfies $\sum_{i=1}^{\ell-1} id_i  \equiv 0 \bmod \ell$ and assume $f=\prod_{i=1}^{\ell-1}f_i^i$ is not a power of an irreducible polynomial. 
	Then for any $\epsilon>0$, there exist positive constants $B_\epsilon$ and $N_\epsilon$ such that the number of superelliptic curves in the form $y^\ell=\prod_{i=1}^{\ell-1} D_i^i(x)$ with $\sum_{i=1}^{\ell-1}\deg(D_i)\leq n$ that admit a dominant map to $C_0$ is at least $B_\epsilon \cdot q^{\frac{2n}{d} - \epsilon}$ for $n>N_\epsilon$.

\end{theorem*}

Let $\mathcal A_\ell(n)$ be the set of degree $\ell$ primitive Dirichlet characters over $\mathbb{F}_q(t)$ whose conductor has norm at most $q^n$ and let $\mathcal B_\ell(n)$ be the subset of $\mathcal{A}_\ell(n)$ containing characters $\chi$ such that $L(1/2,\chi)=0$.

If we assume that $q \equiv 1 \bmod 3$, by the computation in Lemma \ref{primcount} we get 
\[|\mathcal A_3(n)|\sim c_{q}(n+1)q^n. \]
Using Theorem \ref{GeneralThm} we have the following result about the size of $\mathcal B_3(n)$:

\begin{theorem*}[Theorem \ref{trigonalThm}]
	Let $\mathbb{F}_q$ be a finite field of odd characteristic $p$ where $p \equiv 2 \bmod 3$, $q=p^e$, and $e \equiv 0 \bmod 4$. Then for any $\epsilon >0$, there exist positive constants $C_\epsilon$ and $N_\epsilon$, such that  $|\mathcal B_3(n)|\geq  C_\epsilon \cdot q^{\frac{2n}{3} - \epsilon} $ for any $n > N_\epsilon$.
\end{theorem*}

\subsection*{Outline of the paper:}
In section \ref{Background}, we set up notation and background on the relationship between Dirichlet L-functions over function fields and zeta functions of curves. In section \ref{Counting}, we recall an explicit parametrization of $\ell$-th order superelliptic curves and maps between such curves. We use a result of Poonen \cite{Poonen} to prove Theorem \ref{GeneralThm}. In section \ref{MainResults}, we give a discussion on the existence of desired base curve for any $\ell$ to apply Theorem \ref{GeneralThm}. In the case of $\ell=3$ with $p,q$ satisfying the congruence condition in Theorem \ref{trigonalThm}, we apply Theorem \ref{GeneralThm} to an elliptic curve to prove Theorem \ref{trigonalThm}. For other $\ell$, we give an existence result Theorem \ref{ellThm} for $p,q$ satisfying some congruence conditions and $q$ sufficiently large.

\subsection*{Acknowledgments} The authors are grateful to  Patrick Allen, Siegfred Baluyot, Jordan Ellenberg and Soumya Sankar for helpful comments and suggestions. The second author was supported by the Simons collaboration on number theory, arithmetic geometry, and computation. 

\section{Notation and Background}\label{Background}

Let $p$ be an odd prime number and $q=p^e$. Let $\FF_q$ denote the field with $q$ elements, $k=\FF_q(t)$ the field of rational functions on $\FF_q$ in the variable $t$ and $A=\FF_q[t]$, the ring of polynomials. For a polynomial $f\in A$, we define the norm of $f$ as $|f|=q^{\deg(f)}$.

The ring $A$ is a maximal order in $k$. By \emph{function field} we mean a finite extension of $k$. A \emph{constant} extension of a function field $K$ with constant field $\mathbb{F}_q$ is an extension of function fields $L/K$ where $L=K \otimes_{\mathbb{F}_q} \FF_{q^n}$.  An extension of function fields $L/K$ is called \emph{geometric} if $L$ and $K$ have the same constant field $\FF_q$.

Over finite fields, there are various notions of $L$-functions and zeta functions. We now describe the relevant notions for our investigation and their relationships.

\subsection{$L$-functions attached to a Dirichlet character}
Let $f\in A$ be a monic polynomial. Then a \emph{Dirichlet character} of modulus $f$ is a group homomorphism $\chi :(A/fA)^*\to\CC^*$. 
For any multiple $mf$ of $f$, $\chi$ induces a homomorphism  $(A/mfA)^*\to\CC^*$. We call a character $\chi$ of modulus $f$ primitive if it cannot be induced from a modulus of smaller degree and refer to $f$ as the conductor of $\chi$. We can evaluate a Dirichlet character $\chi$ with conductor $f$ at an element $g\in A$ by

$$
\chi(g) =
\begin{cases}
\chi(g\bmod f) & \text{if } f \text{ and } g \text{ are coprime,} \\
0 & \text{else. }
\end{cases}
$$

To a Dirichlet character $\chi$, we attach a Dirichlet $L$-function of a complex variable $s$ by $$L(s,\chi)=\sum_{g}\frac{\chi(g)}{|g|^s}=\prod_{P}\left(1-\chi(P)|P|^{-s}\right)^{-1} $$
where the summation is over all monic polynomials in $A$ and the product is over all monic irreducible polynomials in $A$. 

Let $n \geq\deg(f)$ be an integer. Let $M_n$ denote the set of monic polynomials of $A$ of degree $n$. If $\chi$ is non-principal, the orthogonality relations for $\chi$ imply that the sum $$\sum_{g\in M_n}\chi(g)=0$$ and thus $L(s,\chi)$ is a polynomial in $q^{-s}$ of degree at most $\deg(f)-1$. If $\chi(ag)=\chi(g)$ for any $a\in \FF_q^*$, $g\in \FF_q[t]$, then $\chi$ is said to be \emph{ even}. L-functions associated to even characters always have a trivial zero at $s=0$. For any primitive character $\chi$, set  
\[
    \psi_\infty(\chi) = \begin{cases} 
    1 \text{ if }\chi \text{ is even;}\\
    0  \text{ else.}\end{cases}
\]
    
    Let $L^*(s,\chi):= (1-\psi_\infty(\chi))^{-1}L(s,\chi)$ be the completed $L$-function of $\chi$. Then it is a consequence of the Riemann-Hypothesis for function fields that all zeroes of $L^*(s,\chi)$ have real part equals $\frac{1}{2}$.

\subsection{$L$-functions attached to a Galois character}
A leisurely exposition of the following material can be found in Chapter 9 of \cite{Rosen}.

Let $L$ be a function field. A prime of $L$ is, by definition, a discrete valuation ring $R\subset L$ with maximal ideal $\mathfrak P$ that contains the field of constants of $L$. The \emph{zeta function} of  $L$ is defined as a function of a complex variable $s$: $$\zeta_L(s)= \prod_{\mathfrak P}\left(1-|\mathfrak P|^{-s}\right)^{-1} .$$
Here the product is taken over all the primes of $L$ and $|\mathfrak P|$ is the cardinality of the residue field at the prime $\mathfrak P$ (and is equal to $q^{\deg(\mathfrak P)}$ in the case when $L=\FF_q(t)$ and $\mathfrak P$ is identified with a monic irreducible polynomial in $\FF_q[t]$). 

Fix a finite geometric Galois extension $L/K$ of function fields that is abelian with Galois group $G$. To any prime $\mathfrak P $ lying above a prime $P$ of $K$ that is unramified in $L/K$, denote the Frobenius at $\mathfrak P$, an element of $G$, by Frob$_\mathfrak P$. Denote the group of $\CC$-valued multiplicative characters of $G$ by $\widehat G$. Given a character $\chi\in \widehat G$, we evaluate it at the primes of $L$ via the Frobenius. Namely,  

$$
\chi(\mathfrak P) =
\begin{cases}
\chi(\text{Frob}_\mathfrak{P}) & \text{if } \mathfrak{P}\text{ does not lie above a ramified prime in } L/K, \\
0 & \text{else. }
\end{cases}
$$

The \emph{Artin $L$-function} attached to $\chi$ is defined as an Euler product $$L(s,\chi)=\prod_\mathfrak{P} (1-\chi(\mathfrak{P})|\mathfrak P|^{-s})^{-1} $$
where the product runs over all primes of $L$. Denote by $\chi_0$ the principal character in $\widehat G$. Applying the orthogonality properties of characters to the splitting of primes in Galois extensions, we find the following decomposition $$\zeta_L(s)=\zeta_K(s)\prod_{\chi \neq \chi_0 \in \widehat G}L(s,\chi).$$

In analogy with the classical theory, each $L(s,\chi)$ admits a meromorphic continuation to the entire complex plane and satisfies a functional equation relating it to $L(1-s,\overline \chi)$ which has the following important consequence: $$L(1/2,\chi)=0\iff L(1/2,\overline\chi)=0.$$

We will be concerned with the specific choice of $G=\ZZ/\ell \ZZ$ and restrict to the case $q\equiv 1\pmod \ell$. Under this restriction, the base field contains all the $\ell$-th roots of unity. By Kummer theory, every Galois extension of $k=\FF_q(t)$ with Galois group $\ZZ/\ell \ZZ$ is obtained by adjoining an $\ell$-th root of an element in $k$.  
\subsection{Zeta functions attached to curves}

Let $C$ be a smooth, projective, geometrically integral curve of genus $g$ defined over $\FF_q$. Let $C(\FF_{q^n})$ denote the set of $\FF_{q^n}$-points of $C$. Then we set $T=q^{-s}$ and define the zeta function of $C$ as the following generating series: $$Z(C,T)= \exp\left(\sum_{n\geq 1}\frac{|C(\FF_{q^n})|T^n}{n}\right).$$

Weil \cite{Weil} proved in 1949 that $Z(C,T)$ is a rational function which can be written as $$Z(C,T)=\frac{P(T)}{(1-T)(1-qT)}$$ where $P(T)\in \ZZ[T]$ is a polynomial of degree $2g$. Moreover, $$P(T)=\prod_{i=1}^{2g} (1-\pi_iT) $$ where each $\pi_i$ is an algebraic integer with complex norm $|\pi_i|=\sqrt{q}$ under every complex embedding. $P(T)$ is in fact the characteristic polynomial of the geometric Frobenius acting on the $\ell$-adic Tate module of the Jacobian of the curve, denoted as $J(C)$ and the $\pi_i$ are therefore the eigenvalues under this action.

The category of smooth projective curves over $\FF_q$ with non-constant morphisms to $\mathbb{P}^1_{\mathbb{F}_q}$ is canonically equivalent to the category of geometric extensions of $\FF_q(t)$. The equivalence is realized simply by taking a curve $C$ to its function field $k(C)$. This induces the following relation of zeta functions $$Z(C,q^{-s})=\zeta_{k(C)}(s)$$

\subsection{Relations between $L$-functions}
Although our explicitly stated goal is to compute a lower bound on the number of \emph{Dirichlet} $L$-functions vanishing at $s=\frac{1}{2}$, our methods rely on finding curves whose zeta functions have some special properties. So we now describe a construction that relates the above $L$-functions with the zeta functions of curves. There is, in fact a fourth type of characters that we have not mentioned here. They are called \emph{Hecke characters} with their own attached $L$-functions, but since we only use them as intermediaries between Dirichlet $L$-functions and Artin $L$-functions, we do not discuss them in detail and refer the interested reader to chapter $9$ of \cite{Rosen} for further reading.

For a polynomial $f$, the group $(A/fA)^*$ is canonically isomorphic to the divisor class group of degree zero for the modulus $f\infty$, denoted by $Cl^o_{f\infty}$. Thus any Dirichlet character $\chi$ can be identified with a finite order character of the ray class group of modulus $f\infty$. By class field theory, this arises from a character of the ray class field $k_{f\infty}$ whose Galois group is Gal$(k_{f\infty}/k)\simeq Cl^o_{f\infty}\simeq (A/fA)^*$. 
The kernel of the character $\chi$, now considered as a homomorphism Gal$(k_{f\infty})\to \CC^*$ fixes a field $k_{\chi}\subset k_{f\infty}$ which is Galois over $k$ and has a cyclic Galois group.

We will crucially use the fact that under these identifications, the Dirichlet $L$-function, the Hecke $L$-function and the Artin $L$-function attached to $\chi$ (considered as a Dirichlet character, ray class character and Galois character respectively) are all identical. It is important for $\chi$ to be a primitive Dirichlet character for otherwise the resulting $L$-functions will be missing some Euler factors. 

However, we will need quantitative results connecting the degree of the conductor of the character $\chi$ and the genus of the smooth projective curve with function field $k_\chi$ in the above construction. Hence we will make extensive use of explicit class field theory for $k$ involving cyclotomic function fields.
\begin{definition}{\label{superelliptic}}
    For positive integers $n\geq 2$ and a prime power $q=p^e$ such that $(n,p)=1$, an \emph{$n$-th order superelliptic curve} over $\FF_q$ is a smooth projective curve $C$ whose function field $k(C)$ has the form  $k[y]/\langle y^n-f(t)\rangle$ where $f(t)$ is an $n$-th power free polynomial.
    \end{definition}

For the rest of this section, assume that $n=\ell$ is a prime such that $q \equiv 1 \pmod {\ell}$. This is equivalent to the field $\FF_q$ containing all the $\ell$-th roots of unity. Then every $\ell$-th order superelliptic curve is a branched Galois covering of $\PP_{\FF_q}$ with Galois group $\ZZ/\ell\ZZ$.

Conversely, given a Galois cover of $\PP_{\FF_q}$ with Galois group $\ZZ/\ell \ZZ$, by Kummer theory, the function field  $k(C)=K(\sqrt[\ell]{\alpha})$ for some $\alpha \in \FF_q[t]$ an $\ell$-th power free polynomial. Then $C$ can be defined as the smooth projective model of an affine, possibly singular curve given by equation: $$y^\ell = F_1(x) F_2^2(x) \ldots F_{\ell-1}^{\ell-1}(x) $$ where $(F_1,\ldots,F_{\ell-1})$ are pairwise coprime square-free polynomials of degree $(d_1,\ldots,d_{\ell-1})$ respectively. Of course, more than one such defining equation could correspond to the same (isomorphism class of a) superelliptic curve. From now on, we denote $\sum_{i=1}^{\ell-1}d_i$ by $d$. The genus of the curve $C$ is given by 
\begin{align*}
 g=  &\frac{(d-2)(\ell-1)}{2}.
\end{align*}

We now describe their relationship between primitive $\ell$-th order Dirichlet characters whose L-functions that vanish at $s=\frac{1}{2}$ and $\ell$-th order superelliptic curves whose zeta functions vanish at $s=\frac{1}{2}$.
\begin{itemize}
\item Let $\mathcal S_{\ell}(n)$ be the set of isomorphism classes of $\ell$-th order superelliptic curves of genus $g \leq \frac{(\ell-1)(n-2)}{2}$.
\item Let $\mathcal T_{\ell}(n)$ be the subset of curves $C$ in $\mathcal S_{\ell}(n)$ for which $Z(C,q^{-\frac{1}{2}})=0$. 
\item Let $\mathcal A_{\ell}(n)$ be the set of primitive $\ell$-th order Dirichlet characters with conductor $m$ with  $\deg m \le n$. 
\item Let $\mathcal B_{\ell}(n)$ be the subset of $\mathcal A_{\ell}(n)$ consisting of $\chi$ for which $L(\frac{1}{2},\chi)=0$. 
\end{itemize}
Once $\ell$ is fixed, there do not exist curves $C$ of genus $g$ for arbitrary genus unless $\ell=2$. This latter case is dealt with in \cite{Wanlin} and in this situation there is a bijection between models of hyperelliptic curves $y^2=D$ and primitive quadratic non-principal Dirichlet characters.

The following key lemma lets us transfer our attention from Dirichlet $L$-functions to Zeta functions of superelliptic curves:

\begin{lemma} \label{dirtocurves} If $q\equiv 1 \pmod \ell$, then $|\mathcal B_{\ell}(n)| \geq |\mathcal T_{\ell}(n)|$ for all positive $n$. 
\end{lemma}

\begin{proof}

We construct a surjective map $\varphi: \mathcal A_{\ell}(n) \to \mathcal S_{\ell}(n)$ such that for each $C\in \mathcal T_{\ell}(n)$, there is a $\chi \in \mathcal B_\ell(n)$ such that $\varphi(\chi)=C$.

For each non-constant polynomial $m\in A$, explicit class field theory for $k$ describes a geometric Galois extension $k_m/k$ with Gal$(k_m/k)\simeq (A/mA)^*$.  The field $k_m$ is called the cyclotomic function field associated to the polynomial $m$. 

Given a primitive $\ell$-th order Dirichlet character $\chi$ of conductor $m$, its kernel determines, via the isomorphism Gal$(k_m/k) \simeq (A/mA)^*$, a geometric Galois extension $k_\chi/k$ with Galois group $\ZZ/\ell\ZZ$. Since $k$ contains all the $\ell$-th roots of unity, by Kummer theory $k_\chi = k(\sqrt[\ell]{D})$ for some $\ell$-th power free polynomial $D \in A$, which is the function field of the (possibly singular) affine curve $y^\ell = D$, whose normalization is exactly the $\ell$-th order superelliptic curve $C$ we call $\phi(\chi)$. Then by the Riemann-Hurwitz formula, $$2g-2=l(2g_{\PP}-2)+\deg(\text{Disc}(k(C)/k)) $$ and applying the conductor-discriminant formula for function fields, we find that the genus of the curve and the degree of it's conductor are related by  $$g = \frac{(\ell-1)(n-2)}{2} $$ 

To show that $\varphi$ is surjective, given a curve $C$ in $\mathcal S_{\ell}(g)$, the function field of $C$, $k(C)$, is a geometric abelian extension of $k$ that is unramified or tamely ramified at infinity since all the ramification indices at the ramification points of the covering map $C\to \PP$ divide $\ell$. Again, by the conductor-discriminant formula for function fields, the conductor of $k(C)$ is a monic polynomial $m$ of degree $\frac{2g}{\ell-1}+2$. As a consequence of explicit class field theory for $k$, we have an inclusion $k\subset k(C) \subset k_m$, where $k_m$ is the cyclotomic function field associated to the polynomial $m$. We may correspondingly identify a subgroup  $H\subset \text{Gal}(k_m/k) \simeq (A/mA)^*$ and pick any character . If $\chi$ is not primitive for modulus $m$, we may find an appropriate divisor of $m$ for which it is.

Given a prime $P$ of $k$, it's behavior in the extension $k_m/k$ is determined entirely by it's residue class modulo $m$ (Theorem 12.10 of \cite{Rosen}). As a consequence, the character $\chi$ has the same $L$-function whether it is interpreted as a Dirichlet character on $(A/mA)^*$ or a Galois character on Gal$(k_m/k)$. Since the zeta function $Z(C,u)$ decomposes as a product of $L$-functions of the Galois group Gal$(k_m/k)$ and the zeta function of $\PP$, we find that the condition $Z(C,q^{-1/2})=0$ forces the Dirichlet $L$-function associated to one of the characters $\chi\in \mathcal A_\ell(n)$ to vanish at $s=\frac{1}{2}$. Thus the map $\varphi$ has the claimed property and the lemma follows.

\end{proof}

As we are most interested in finding a lower bound on $|B_\ell(n)|$, we restrict our attention to study $\ell$-th order superelliptic curves and speak no further of Dirichlet characters, thanks to Lemma 2.2.

In the next section, we study models of $\ell$-th order superelliptic curves $C$ for which $Z(C,q^{-1/2})=0$. Since we are working with isomorphism classes of curves, we can use models for $\ell$-order superelliptic curves that are not ramified at $\infty$. These models satisfy: $$\sum_{i=1}^{\ell-1}id_i \equiv 0 \bmod \ell .$$
\subsection{Counting Primitive Characters}

We now count the number of primitive Dirichlet characters with conductor of fixed degree, as well as primitive Dirichlet characters of a fixed order. For this section, let  $\zeta_q(s) $ denote the zeta function of $\FF_q(t)$.

\begin{lemma}
Denote the number of primitive Dirichlet characters $\chi$ whose conductor is a polynomial of degree $d$ by $\mathcal A(d)$. Then, \[ \mathcal A(d)=\begin{cases} 
      1 & d= 0 \\
      q^2-2q & d=1 \\
      q^{2d-2}(q-1)^2 & d\geq 2 
   \end{cases}
\]

\end{lemma}
\begin{proof}
For a monic polynomial $f\in \FF_q[t]$, let $\Phi(f)$ denote the number of characters of modulus $f$ and let $Q(f)$ denote the number of primitive characters with conductor $f$. Since abelian groups are isomorphic to their character groups, we have $\Phi(f)=|(\FF_q[t]/f)^*|$. A consideration of the possible conductors of a character reveals that the functions $\Phi(f)$ and $Q(f)$ are related by:
    $$\Phi(f)= \sum_{g|f, \text{monic}}Q(g).$$
The above convolution leads to a relation among the corresponding Dirichlet series: $$\sum_{f \text{ monic}}\frac{\Phi(f)}{|f|^s} = \zeta_q(s)\sum_{f \text{ monic }}\frac{Q(f)}{|f|^s} $$
Here $\zeta_q(s)$ is the affine zeta function of the ring $\FF_q[t]$. It differs from the zeta function of $\FF_q(t)$ by a single factor of $\frac{1}{1-q^{-s}}$. Using the substitution $u=q^{-s}$, $\zeta_q(s)$ takes the form $(1-qu)^{-1}$. The left hand side is given by $\frac{\zeta_q(s-1)}{\zeta_q(s)}$, which is $\frac{1-qu}{1-q^2u}$. We then reorganize the following expression by grouping together factors of the same degree:
$$\sum_{f \text{ monic }}\frac{Q(f)}{|f|^s}=\sum_{d=0}^\infty \left(\sum_{\substack{f \text{ monic} \\ \deg(f)=d}}Q(f)\right)u^d $$
The parenthetical expression is the number of primitive characters with conductor of degree $d$, denoted $\mathcal A(d)$, It satisfies $$\sum_{d=0}^\infty \mathcal A(d)u^d = \frac{\zeta_q(s-1)}{(\zeta_q(s))^2}=\frac{(1-qu)^2}{(1-q^2u)}=(1-2uq+q^2)\sum_{d=0}^\infty (q^2u)^d$$
Comparing the coefficients of $u^d$, we obtain the stated formula.
\end{proof}

Since we study $L$-functions attached to characters of a fixed order, we now count the number of characters of order $\ell$ whose conductor has prescribed degree $d$.

\begin{lemma}
\label{primcount}
Fix odd primes $p, \ell$ and $q=p^n$ such that $q\equiv 1 \pmod \ell$. Let $\mathcal A_\ell(d)$ denote the number of Dirichlet characters of order $\ell$ whose conductor has degree $d$. Then there exists a positive constant $c_{q,\ell}$ such that $\mathcal A_\ell(d)\sim c_{q,\ell}q^dd^{\ell-2} $.
\end{lemma}
\begin{proof}
For this proof, denote $A=\FF_q[t]$. For a monic element $f\in A$, the characters of modulus $f$ of order $\ell$ along with the trivial character are given by Hom$((A/fA)^*,\ZZ/\ell\ZZ)$.
We can decompose $f$ into a product of primes $\prod_{i=1}^rP_i^{e_i}$, where $r=\omega(f)$, the number of distinct monic prime factors of $f$. For any prime $P\in A$, the group $(A/P^eA)^* $ is given as an extension of the cyclic group $(A/P)^*$ by a $p$-group $G$:
$$1 \to G \to  (A/P^eA)^*\to (A/P)^*\to 1 $$
As there are no non-trivial maps from $G$ to $\ZZ/\ell\ZZ$, every map in Hom$((A/P^eA)^*,\ZZ/\ell\ZZ)$ factors through $(A/P)^*$. We conclude that every primitive character of order $\ell$ has square-free conductor and also that $$|\text{Hom}((A/fA)^*,\ZZ/\ell\ZZ)|=\prod_{i=1}^{r} |\text{Hom}(A/PA,\ZZ/\ell\ZZ)|=\ell^{r} $$

Thus, for a square-free polynomial $f=\prod_{i=1}^r P_i$, there are $(\ell-1)^r$ primitive characters with conductor $f$. In particular, if $\ell=2$, this recovers the fact that for every square-free polynomial $f$, there is a unique quadratic character with conductor $f$. It follows that the generating series $\mathcal G_\ell(u)$ for the number of primitive characters of order $\ell$ is given by $$\mathcal G_\ell(u)=\sum_{d=0}^\infty\mathcal A_\ell(d)u^d=\prod_{P}(1+(\ell-1)u^{\deg P}), $$
where $\mathcal A_\ell(d)$ is the number of primitive order $\ell$ characters with conductor of degree $d$. 

Let $\zeta_q(u)=\prod_{P}(1-u^{\deg P} )^{-1}$ be the zeta function of $\FF_q(t)$. It converges when $|u|<\frac{1}{q}$ and is analytically continued to the entire complex plane by the function $\frac{1}{1-qu}$ with a single simple pole at $u=\frac{1}{q}$. Since $$\zeta_q(u)^{\ell-1}=\prod_{P}(1-u^{\deg P} )^{-(\ell-1)}=\prod_{P}(1+(\ell-1)u^{\deg P}+O(u^{2\deg P} )),$$ it follows that $\mathcal G_\ell(u)$ can be analytically continued to $|u|\leq \frac{1}{q}$ with only a single pole of order $\ell-1$ at $u=\frac{1}{q}$. By a standard Tauberian theorem, e.g. Theorem 17.4 of \cite{Rosen}, we conclude that there is a positive constant $c_{q,\ell}$ such that $$\mathcal A_\ell(d)\sim c_{q,\ell}q^dd^{\ell-2}.$$
\end{proof}
\section{Maps Between $\ell$-th order superelliptic curves}\label{Counting}

\begin{lemma}\label{C0toC}
Let $C_0$ be a curve over $\FF_q$ with zeta function $Z(C_0,q^{-s})$. If $Z(C_0,q^{-s_0})=0$ for some $s_0\in\CC$, then the zeta function of any curve $C$ defined over $\FF_q$ admitting a dominant map to $C_0$ also satisfies  $Z(C,q^{-s_0})=0$.	
\end{lemma}

\begin{proof}
	The dominant map from $C$ to $C_0$ induces an isogeny from the Jacobian $J(C_0)$ to an isogenous factor of $J(C)$. So the eigenvalues of the Frobenius action on the $\ell$-adic Tate module of $J(C_0)$ also appear as Frobenius eigenvalues of $T_\ell J(C)$. As the roots of the zeta function of a curve exactly correspond to these Frobenius eigenvalues, the result follows.
\end{proof}

Lemma \ref{C0toC} allows us to construct infinitely many curves whose zeta function admits some value as a root from the existence of one such curve. 
Thus, if we could show the existence of one $\ell$-th order superelliptic curve $C_0$ admitting a specific Frobenius eigenvalue (namely $q^\frac{1}{2}$, but the specific eigenvalue is not important here), then
 to give a lower bound on the number of all $\ell$-th order superelliptic curves admitting that specific Frobenius eigenvalue, it is enough for us to provide a lower bound on the number of $\ell$-th order superelliptic curves with genus less or equal to $g$ which admit a dominant map to $C_0$. To do this, we give an explicit construction of such a dominant map.

\begin{lemma}\label{constructionC}
Let $C_0$ be the smooth projective model of the curve given by affine equation $$y^\ell=f_1f_2^2 \ldots f_{\ell-1}^{\ell-1}$$ where each $f_i\in \mathbb{F}_q[x]$ is square-free. If $C$ is the smooth projective model of the curve defined by $$y^\ell = f_1(h(x))f_2^2(h(x))\ldots f_{\ell-1}^{\ell-1}(h(x)) $$ for a non-constant rational function $h(x) \in \mathbb{F}_q(x)$, there exists a dominant map from $C$ to $C_0$.
\end{lemma}

\begin{proof}
A dominant map $\psi: C \rightarrow C_0$ is given by $(x,y) \mapsto (h(x),y)$.
\end{proof}

Fix an $\ell$-th order superelliptic curve $C_0$ of genus $g_0$ with a defining equation $$y^\ell = f_1(x)f_2^2(x)\ldots f_{\ell-1}^{\ell-1}(x).$$  and consider the set $G(C_0,g)$ of models of $\ell$-th order superelliptic curves $C$ with $g(C)  \leq g$ and admitting a dominant map $\phi: C \rightarrow C_0$.

To give a lower bound for the size of $G(C_0, g)$, using Lemma \ref{constructionC} it suffices to count the number of curves of bounded genus, admitting a defining equation of the form $$y^\ell = f_1(h)f_2^2(h)\ldots f_{\ell-1}^{\ell-1}(h)$$ where $$h(x)=p(x)/q(x) \text{ for some } p(x),q(x) \in \mathbb{F}_q[x].$$ Define $$\deg h = \max \{ \deg p, \deg q \}.$$ If $$\deg h \le (g+\ell-1)/(g_0+\ell-1),$$ then it is guaranteed that the genus of $C$ is bounded by $g$.

Define $F_1,F_2,\ldots,F_{\ell-1}$ to be the homogenized polynomials for $f_1,f_2,\ldots,f_{\ell-1}$, i.e. $$F_i(p,q)=q^{\deg f_i} f_i(p/q)$$ and 
\[F= \prod_{i=1}^{\ell-1} F_i.\]

\begin{definition}
	Let $n=\frac{2g}{\ell-1}+2$. Define set
	\begin{align*}
	    P(n) =
		\{ &(D_1, \ldots, D_{\ell-1}) \in (\mathbb{F}_q[t])^{\ell-1} :\\
		&D_1,\ldots,D_{\ell-1} \text{ are pairwise coprime, monic, square-free, } \deg (D_1\cdots D_{\ell-1})\leq n \}.
	\end{align*}
\end{definition}

To give a lower bound on $|G(C_0,g)|$, it suffices to count the number of tuples\\ $(D_1,\ldots, D_{\ell-1}) \in P(n)$ such that there exists $(p,q) \in (\mathbf{F}_q[t])^2$ where
\begin{align}\label{Pickingpq}
	D_1 = F_1(p,q),
	D_2 = F_2(p,q),
	\ldots ,
	D_{\ell-1} = F_{\ell-1}(p,q).
	\end{align}

We will do this in two steps. 

First, we obtain a lower bound on the number of pairs $(p,q) \in (\mathbf{F}_q[t])^2$ with $$\max \{ \deg p, \deg q \} \le (g+\ell-1)/(g_0+\ell-1)$$ where the product
$F(p,q) = F_1(p,q)F_2(p,q)\ldots F_{\ell-1}(p,q) $ is square-free. 

Next, for a fixed tuple $(D_1,D_2,\ldots, D_{\ell-1}) \in P(n)$, we give an upper bound on the number of pairs $(p,q) \in (\mathbf{F}_q[t])^2$ such that Equations \ref{Pickingpq} are satisfied.

For the first step, we need the following result.

\begin{proposition}\label{poonen}\cite[Theorem 8.1]{Poonen}
	Let $P$ be a finite set of primes in $\mathbb{F}_q[t]$, $B$ be the localization of $\mathbb{F}_q[t]$ by inverting the primes in $P$, $K = \mathbb{F}_q(t)$, $f \in B[x_1, \ldots , x_m]$ be a polynomial that is square-free as an element of $K[x_1, \ldots , x_m]$ and for a choice of $x \in \mathbb{F}_q[t]^m$, we say that $f(x)$ is square-free in $B$ if the ideal $(f(x))$ is a product of distinct primes in $B$. For $b \in B$, define $|b| = |B/(b)|$ and for $b=(b_1,\ldots, b_n) \in B^n$, define $ |b| = \max{|b_i|}$. Let
	\begin{align*}
	S_f &:= \{ x \in \mathbb{F}_q[t]^m : f(x) \text{ is square-free in } B \},\\
	\mu_{S_f} &:= \lim\limits_{N \rightarrow \infty} \frac{|\{ b \in S_f : |b|<N \}|}{N^m}. 
	\end{align*}
	For each nonzero prime $ \pi$ of $B $, let $c_\pi$ be the number of $ x \in ( A / \pi^2 )^m $ that satisfy $f(x) = 0$ in $  A / \pi^2 $. 
	The limit  $\mu_{S_f}$ exists and is equal to $ \prod_{\pi}(1-c_\pi / |\pi|^{2m})$.
\end{proposition}

\begin{proof}
	This proposition directly follows from Theorem 8.1 of \cite{Poonen} by setting the ``box" to be \\$\{ x_1,\ldots,x_m \in \mathbb{F}_q[t]: |x_i|<N  \}.$
\end{proof}

	\begin{remark}\label{defineA}
	For our purpose, we will take $f=F(p,q)$ and it is crucial to have $\mu_{S_f}>0$. In order to ensure this, it suffices to check that none of the factors $(1-c_\pi / |\pi|^{2m})$ is zero. We take $m=2$ for our case.
	If for some prime $\pi$ in $\mathbb{F}_q(t)$, $1-c_\pi / |\pi|^{4}=0$, then this means $$F(u,v) \equiv 0 \mod \pi^2$$ for all $ (u,v) \in (\mathbb{F}_q[t])^2$. Thus $F(u,v) \equiv 0 \mod \pi$ for all $ (u,v) \in (\mathbb{F}_q[t]/\pi)^2$. Since the coefficients of $F(u,v)$ are units in $\mathbb{F}_q[t]$, it must be the case that $(F \mod \pi)$ is not the zero polynomial. 
	
	This implies $(F \mod \pi)$ can at most have $ d |\mathbb{F}_q[t]/\pi|$ solutions in $(\mathbb{F}_q[t]/\pi)^2$ where $d=\deg F$. So $$ d |\mathbb{F}_q[t]/\pi| \ge |\mathbb{F}_q[t]/\pi|^2$$ which is equivalent to $|\pi|\leq d$.
	
	Thus, we choose $P_f$ be the set of primes $\pi$ of $\mathbb{F}_q(t)$ such that $|\pi| \le d $. Let $B$ be the localization of $\mathbb{F}_q[t]$ by inverting all the primes in $P_f$. This implies $1-c_p / |p|^{4} \ne 0$ for any prime $p$ in $B$, and thus neither is the infinite product $\mu_{S_f}$.
	\end{remark}

\begin{theorem}\label{GeneralThm}
	Let $\ell$  be an odd prime number coprime to $q$. Let $C_0$ be an $\ell$-th order superelliptic curve of genus $g$ defined over $\mathbb{F}_q$ with affine equation $y^\ell=f_1(x)f_2^2(x)\ldots f_{\ell-1}^{\ell-1}$ where the $f_i$ are pairwise coprime, squarefree polynomials of degree $d_i$ each. Pick a model such that $\sum_{i=1}^{\ell-1} id_i  \equiv 0 \bmod \ell$ and assume $f_1(x)f_2^2(x)\ldots f_{\ell-1}^{\ell-1}$ is not a power of an irreducible polynomial. 
	Then for any $\epsilon>0$, there exist positive constants $B_\epsilon$ and $N_\epsilon$ such that the number of tuples of polynomials $(D_1, \dots ,D_{\ell-1}) \in P(n)$ satisfying the condition that curve $C :y^\ell = D_1(t) D_2(t)^2 \ldots D_{\ell-1}(t)^{\ell-1}$ admits a dominant map to $C_0$ is at least $B_\epsilon \cdot q^{\frac{2n}{d} - \epsilon}$ for $n>N_\epsilon$ where $d=d_1+\ldots + d_{\ell-1}$.
\end{theorem}

\begin{proof}
	Consider curves $C$ satisfying the condition in Lemma \ref{constructionC}. As was discussed earlier in this section, to give a lower bound on the number of such curves, it suffices to give a lower bound on the number of tuples $(D_1(t),D_2(t),\ldots, D_{\ell-1}(t)) \in P(n)$ such that there exists  $(u(t),v(t)) \in (\mathbf{F}_q[t])^2$ and
	\begin{equation}\label{EquationsDiFi}
	D_1(t) = F_1(u(t),v(t)),
	D_2(t) = F_2(u(t),v(t)),
	\ldots, 
	D_{\ell-1}(t) = F_{\ell-1}(u(t),v(t)).
	\end{equation} 
	
	Denote by $N=q^n$.
	By Proposition \ref{poonen}, for $(u,v) \in (\mathbf{F}_q[t])^2$ with $\{\deg u, \deg v \} \le n/d$, there are $ \gg \mu_{S_{F_1F_2\ldots F_{\ell-1}}} N^{2/d} $ such pairs that satisfy the condition that $$F_1(u(t),v(t))F_2(u(t),v(t))\ldots F_{\ell-1}(u(t),v(t))$$ is a squarefree element in $\mathbb{F}_q[t]$ and $\mu_{S_{F_1F_2\ldots F_{\ell-1}}}>0$ by taking $B$ as defined in Remark \ref{defineA}. 
	
	To conclude, for each fixed tuple $(D_1,\ldots, D_{\ell-1}) \in P(n)$, we need an upper bound on the number of pairs $(u(t),v(t))$ such that Equations \ref{EquationsDiFi} hold to correct doublecount. 
	
	First, we consider the case where there exist integers $i,j$ with $1\le i < j\le \ell -1$ such that $F_i$, $F_j$ are both non-constant. Without loss of generality, we assume $i=1$ and $j=2$. Thus, it suffices to bound the number of $(u,v)$ such that $D_i=F_i(u,v)$ for $i=1,2$. Geometrically, this can be viewed as the intersection locus of two affine curves in $\mathbb{A}^2_{\mathbb{F}_q(t)}$. As $F_1,F_2$ are coprime, this intersection is $0$ dimensional.	By Bezout's theorem, there are at most $\deg F_1 \deg F_2$ sets of solutions for the pair of equations.
	
	Now the case left is when there is only one nonzero $F_i$. Again, we assume $i=1$. Then by our assumption $F_1$ is reducible and take a factorization as $F_1=F_{1,1}F_{1,2}$. Then for a fixed $D_1$ with $\deg D_1 =n$, it has $\ll N^\epsilon$ factors where $N=q^n$. Then for each unit $a$ and a factorization $D_1=D_{1,1}D_{1,2}$, we want to bound the number of $(u,v)$ such that $aD_{1,1}=F_{1,1}(u,v)$ and $a^{-1}D_{1,2}=F_{1,2}(u,v)$. Again, by Bezout's theorem, there are at most $\deg D_{1,1}D_{1,2}$ sets of solutions. 
	
	In any case, for a fixed tuple $(D_1,\ldots, D_{\ell-1}) \in P(n)$, the number of pairs $(u,v)$ such that Equations \ref{EquationsDiFi} hold is bounded above by $qn^2N^\epsilon$ which concludes the proof.
\end{proof}

\begin{remark}
There is a distinction to be made between the collection of models of $\ell$-th order superelliptic curves and the collection of isomorphism classes of $\ell$-th order superelliptic curves. However, each such isomorphism class can contain at most $\ell(\ell-1)|\text{PGL}_2(\FF_q)|$ models. See \cite{Sankar}, Section 2.2 for details. Thus these two sets differ only by a multiplicative factor depending only on $\ell$ and $q$ which are fixed at the start. Since we are content to compute the order of magnitude of a lower bound on the number of curves whose zeta function vanishes at $s=\frac{1}{2}$, counting curves versus counting models does not affect our answer. So the statement of Theorem \ref{GeneralThm} is equivalent to its description in the introduction.
\end{remark}

\section{Theorem on Curves}\label{MainResults}
We finally apply the general results of section \ref{Counting}  to studying the vanishing of zeta functions of curves at $s=\frac{1}{2}$. We recall the notations used in the introduction
	\begin{align*}
	\mathcal A_\ell(n) &:= \{\text{Primitive Dirichlet characters } \chi_f \text{ of order }\ell \text{ with }   \deg f \leq n\},\\
    \mathcal B_\ell(n)&= \{\chi\in \mathcal A_\ell(n) \text{ such that } L(1/2,\chi)=0 \}.
	\end{align*}

\begin{theorem}\label{trigonalThm}
	Let $\mathbb{F}_q$ be a finite field of odd characteristic $p$ where $p \equiv 2 \bmod 3$, $q=p^e$, and $e \equiv 0 \bmod 4$. Then for any $\epsilon >0$, there exist positive constants $C_\epsilon$ and $N_\epsilon$, such that  $|\mathcal B_3(n)|\geq  C_\epsilon \cdot q^{\frac{2n}{3} - \epsilon} $ for any $n > N_\epsilon$. The constants $C_\epsilon$ and $N_\epsilon$ also depend on $q$.
\end{theorem}

\begin{proof}
	Let $E$ be the elliptic curve over $\mathbb{F}_p$ given by Weierstrass equation $y^2=x^3+1$. Since $p \equiv 2 \bmod 3$, $E$ is supersingular with Frobenius eigenvalues $i\sqrt{p}$ and $-i\sqrt{p}$. Thus, consider $E \times_{\mathbb{F}_p}  \mathbb{F}_q$, it has Frobenius eigenvalue $\sqrt{q}$ with multiplicity $2$.
	
	Note that $E$ admits a degree $3$ cyclic cover of $\mathbb{P}^1$ by $(x,y) \mapsto y$. Explicitly, by a change of variable, $E$ has defining equation $y^3 = x^3-x$ and is therefore an $\ell$-th order superelliptic curve with $\ell=3$.
	
	Combining Lemma \ref{C0toC} and Theorem \ref{GeneralThm} by taking $C_0$ to be $E$, we obtain that there are $\gg C_\epsilon q^{\frac{2n}{3}-\epsilon}$ models of superelliptic curves. By Lemma \ref{dirtocurves}, this is also a lower bound on $|\mathcal B_3(n)|$.
\end{proof}
More generally, for an arbitrary odd primes $\ell$, we have the following result for cyclic $\ell$ covers over $\FF_{p^d}$ for some finite fields of odd characteristic:

\begin{theorem}\label{ellThm}
	Let $\ell>2$ be a prime and $p\equiv -1 \bmod \ell$ a prime number. Then, there exists an integer $d>0$  (depending on $\ell$ and $p$) such that when we consider curves over the field $\FF_{q^d}$, there exists a positive constant $N$ such that  $|\mathcal B_\ell(n)|\geq B\cdot q^{\frac{2n}{3}-\epsilon}$ whenever $n\geq N$.
\end{theorem}

\begin{proof}
By remark 3.4 of \cite{LMPT1}, curve $C_0$ given by affine equation $y^\ell = x(x-1)(x-2)^{\ell-2}$ over $\mathbb{F}_p$ is supersingular and is a cyclic $\ell$-cover of $\mathbb{P}^1$ after base change to a field $\mathbb{F}_q$ where $q \equiv 1 \bmod \ell$. Since the Frobenius eigenvalues of $C_0$ considered as a curve over $\mathbb F_{p}$ are all of the form $\sqrt{p}\zeta$ where $\zeta$ is a root of unity, after a base change to a finite extension $\mathbb F_{p^d}$,  $C\times \mathbb F_{p^d}$ acquires $\sqrt{p^d}$ as an eigenvalue. By construction, it has the form $y^l=f(x)$ where $f=x(x-1)(x-2)^{\ell-2}$ is not a power of an irreducible polynomial. Applying Theorem \ref{GeneralThm} combined with lemma \ref{dirtocurves}, our result follows.  
\end{proof}

\bibliographystyle{amsplain}
\bibliography{characterbib}

\end{document}